\documentclass[10pt]{article} 
\usepackage{amsmath, amssymb, latexsym}       
\usepackage{graphicx}
\usepackage{palatino}                     

\begingroup 
\newtheorem{theorem}{Theorem}[section]          
\newtheorem{lemma}[theorem]{Lemma}
\newtheorem{proposition}[theorem]{Proposition}
\newtheorem{corollary}[theorem]{Corollary}
\newtheorem{definition}[theorem]{Definition}
\newtheorem{remark}[theorem]{Remark}            
\endgroup

\numberwithin{equation}{section}

\def\VVV{\mathbb{V}}
\def\EEE{\mathbb{E}}
\def\N{{\bf N}}

\def\fd{/\!\!/}
\def\qed{\qquad\framebox[7pt]\medskip\noindent}
\newenvironment{proof}{{\sl Proof:}\quad}{\qed\\ \noindent}
\title{Harmonic evolutions on graphs}

\author{Jerzy Kocik\\ 
{\small Department of Mathematics}\\ 
{\small Southern Illinois University}\\
{\small Carbondale, IL 62901, USA}\\  
\emph{ \small jkocik@siu.edu} \\  }
\date{}

\begin{document} 

\maketitle

\begin{abstract}
We define the harmonic evolution of states of a graph by iterative
application of the harmonic operator (Laplacian over $Z_2$). 
This provides graphs with a new geometric context 
and leads to a new tool to analyze them. 
The digraphs of evolutions are analyzed and classified.
This construction can also be viewed as a certain topological
generalization of cellular automata.

\end{abstract}

\section{Introduction}

The Laplacian of a graph is a well known tool to study graphs.   
Typically, we retrieve its eigenvalues and deduce certain 
graph's properties \cite{AM, GR, Tut}. 
In this paper we rather analyze the graph's ``harmonic operator", 
the Laplacian modulo $\mathbf{Z}_2$, or --- more concretely --- the monoid it generates. 
Interestingly, 
the emerging algebraic properties may be interpreted pictorially and visualized.  
We define a cellular automaton on an arbitrary (simple) graph as follows.
\\
\\
{\bf The Harmonic Game (``game of light")}.  Choose a graph to represent your ``{\bf universe}".  
Each vertex may be in  one of the two states: ``on" or ``off" (excited or not excited).  
Any initial state of a graph will evolve stepwise according to certain ``{\bf laws}":

\begin{enumerate}
\item If a non-excited cell has an odd number of excited neighbors, it becomes excited.
\item If an excited cell has an odd number of non-excited neighbors, it will remain excited, 
otherwise it relaxes and goes to rest.
\end{enumerate}

Examples of evolutions on various graphs may be inspected in  Figures 1 and 2.
\\

%
\begin{figure}
\[
  \includegraphics[width=5in]{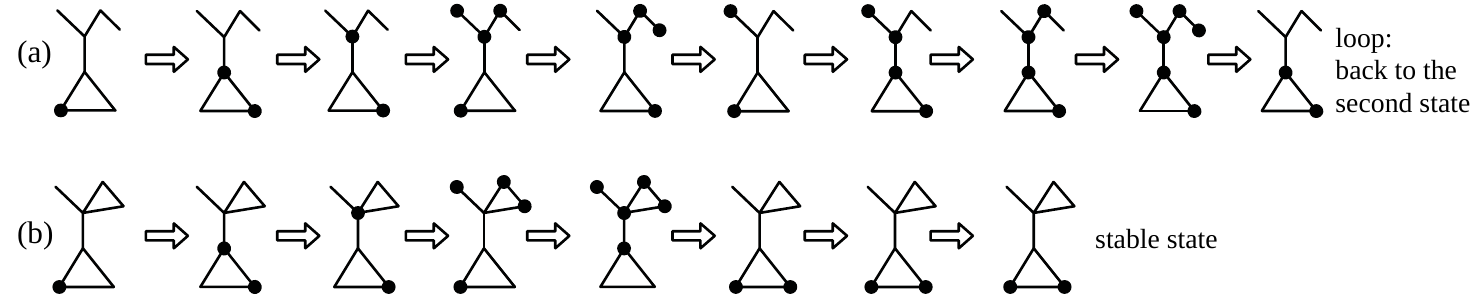}
\]
\caption{Examples of harmonic evolution on graphs (``game of light").}
\label{fig:f1}
\end{figure}
Clearly, each process has to eventually end up in a loop,   
consisting possibly of a single element, like Fig.\ 1b.  
Even small change in structure may radically change the size of the loop and of
the path towards the loop (cf. Fig. 1a and 1b).  
\\

The question is --- what are all possible evolutions on a given graph $G$ due 
to different initial states.  Laborious exercise would recover a digraph 
whose nodes are the states ($2^n$ of them, $n$=number of vertices in $G$), 
and whose directed edges indicate the consecutive states.  
See Figure 3  for an example.
It will be called the {\bf evolution digraph} of $G$ and denoted $G^*$. 
%
\begin{figure}[h]
\[
  \includegraphics[width=4in]{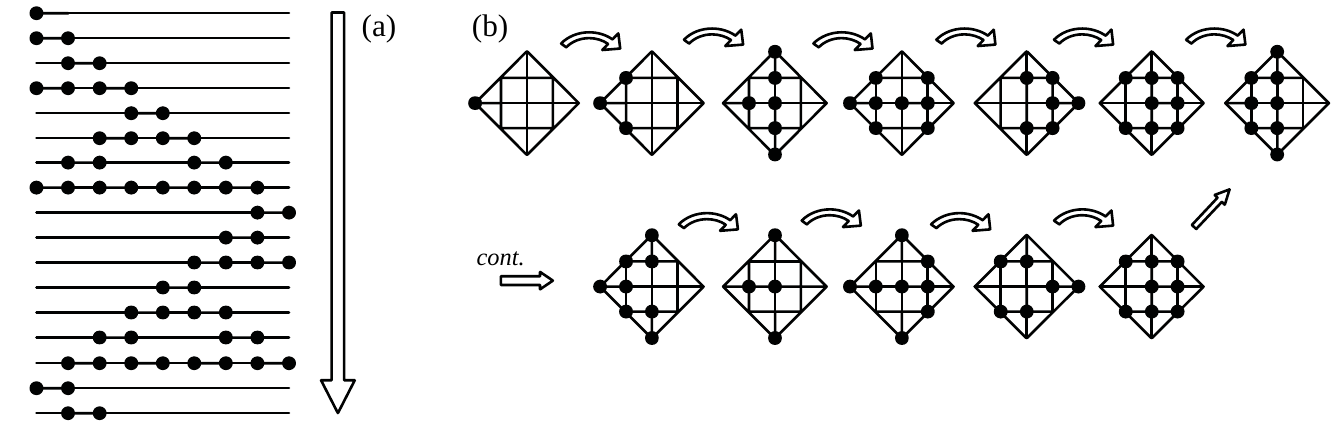}
\]
\caption{Game is like a propagation (a) a string with 9 nodes 
          (b) a diamond shaped graph.}
\label{fig:f2}
\end{figure}
%
It turns out that evolution digraphs have the same general structure: 
each $G^*$ consists of a certain number of closed loops,
and to every node in a loop a certain tree is attached, 
constituted by states descending towards it. 
Interestingly, every tree in $G^*$ has the same form.  
Other examples of graphs and their evolution digraphs are shown in Figure 5. 
\\

The structure of these trees and of the loops is concealed in the harmonic 
operator and its powers.  We will describe an algorithm that uncovers it. 
\\
%
\begin{figure}[h]
\[
  \includegraphics[width=4.4in]{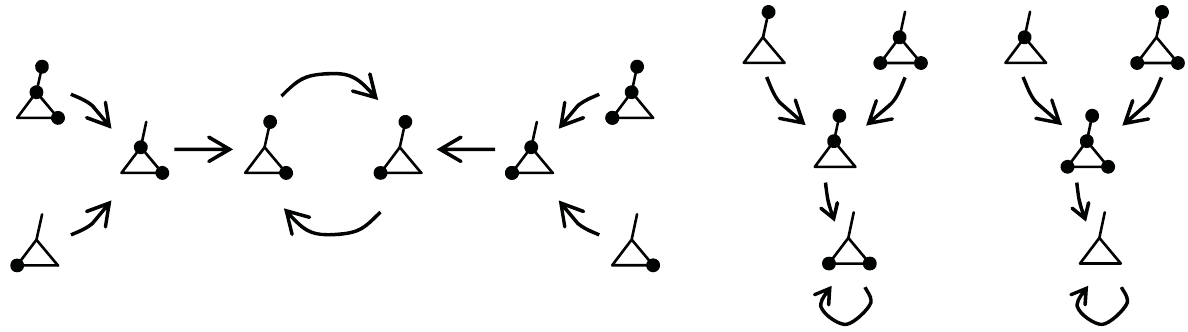}
\]
\caption{Evolution digraph $G^*$ of a graph $G$.}
\label{fig:diG}
\end{figure}

The harmonic game reminds one of Conway's ``game of life."  
But things are here more fundamental.  
Unlike the usual cellular automata, the rules are not arbitrary, but are 
derived from the topologically invariant notions, that of boundary and coboundary. 
They represent the action of the harmonic operator $a$. 
Moreover, the rules and the processes in Figures 1 and 2  
present only the surface of the actual dynamic, because both the vertices and the 
edges are actually involved in the process.  
\\

Motivations for studying Harmonic games (HG):
\\
\\
1. (Graph theory) A tool to analyze graphs in an alternative way 
to that of Laplacian's eigenvalues. The behavior of the states convey 
information on eigenspaces of powers of $a$.  In a sense the harmonic game is 
a method to illustrate the structure of eigenspaces of 
the Laplacian by trees, loops, etc.  
\\
\\
2. (CA) Although we see it as a graph-theoretic study, we 
view HG as a step beyond the regularity of grids used in cellular automata. 
And yet another difference: HG conceals actually two intertwined evolutions, 
one on vertices and one on edges, as explained in the next section 
(the rules presented here present only a surface of the dynamics). 
\\
\\
3. (Metaphor) A toy model of reality.  HG is a metaphor of a typical theory of 
theoretical physics.  
(The graph itself is like a {\it protogeometry}: locations and connections only). 
One is reminded of Wheeler's search for the ``Law without Law" \cite{Whe1} 
or his ``pre-geometry" \cite{Whe2}.   

~
\\
{\bf Remark:} If biologists have their mathematical toy in the ``game of life", 
HG may be viewed as a physicist's toy, a metaphor for ``light propagation":  
The rules for a one-step change defined by the harmonic operator are  
styled after the harmonic operator in differential geometry (and Maxwell's theory of light).  
Harmonic evolution has indeed the flavor of light ``propagation" (see Fig.\ \ref{fig:f2}). 
It actually obeys ``Huygens principle": the sum (mod 2) of two evolutions is an evolution. 
Note other features: the superposition principle, periodicity induced by boundary conditions, 
duality in propagation (like in electromagnetic wave), 
simplicity of laws, their topological foundations, etc.

\section{Harmonic Evolution}

Let $G=(\hbox{Ver\,}(G), \hbox{Edg\,}(G))$ be a simple graph 
(no loops, no multiple edges \cite{Tut}) with the set 
of vertices $\hbox{Ver}(G)$ and the set of edges $\hbox{Edg}(G)$.  
We define a {\it state} (respectively {\it co-state}\,)
as a subset of vertices (respectively edges) of $G$.
Denote the family of all subsets of $\hbox{Ver}(G)$ by $\VVV_G$,  
and that of $\hbox{Edg}(G)$ and by $\EEE_G$.
Define addition as the symmetric difference $A+B=: (A-B)\cup (B-A)$ 
for any two subsets $A$ and $B$ (``adding modulo two").
\\

The {\it boundary} operator $\partial:\VVV_G\to\EEE_G$ 
sends maps into co-states, namely
for any $S\in\VVV_G$, we define $\partial S$ as the sum of 
the sets of edges adjacent to vertices of $S$, added modulo two.
Similarly, a {\it co-boundary} operator 
$\delta:\EEE_G\to\VVV_G$ maps co-states into states, namely 
for $\sigma\in\EEE_G$ we define
$\delta\sigma$ as the set of vertices adjacent to the edges of 
$\sigma$, added modulo two (see Figure~\ref{fig:dd}).
%
%
\begin{figure}[h]
\[
  \includegraphics{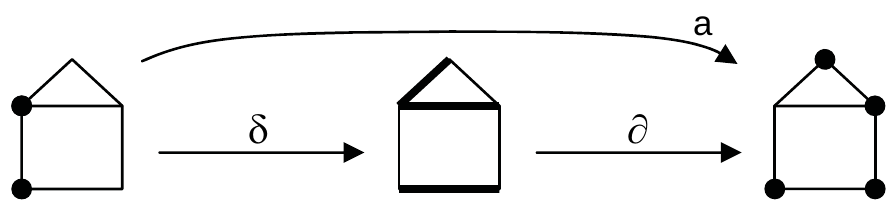}
\]
\caption{Boundary $\partial$, co-boundary $\delta$, 
         and harmonic map $a=\partial\delta$}
\label{fig:dd}
\end{figure}
%
%

By analogy to Hodge theory in differential geometry, 
the {\it harmonic map} is defined as the
composition $a=\partial\delta+\delta\partial$.
In particular, if restricted to states, the harmonic operator becomes
\begin{equation}
\label{eq:a}
                 a=\partial\delta\ \in\hbox{\rm End}\,\VVV_G
\end{equation}
By iterative application of the harmonic operator to an initial state 
$S\in\VVV_G$, we obtain a sequence $\{S,aS,a^2S,\ldots\}$ that we call 
the {\it evolution} of $S$, examples of which are 
in Fig. \ref{fig:f1} and \ref{fig:f2}.
\\

For a given graph $G$, all evolutions form an {\it evolution digraph}
$G^*$, the nodes of which are the states, $\VVV_G$, and the
oriented edges of which are of the form $(S,\;aS)$ (see Fig.\ \ref{fig:diG}).
The problem is to determine the evolution digraph of a given graph.
\\

%
\begin{figure}[h]
\[
  \includegraphics{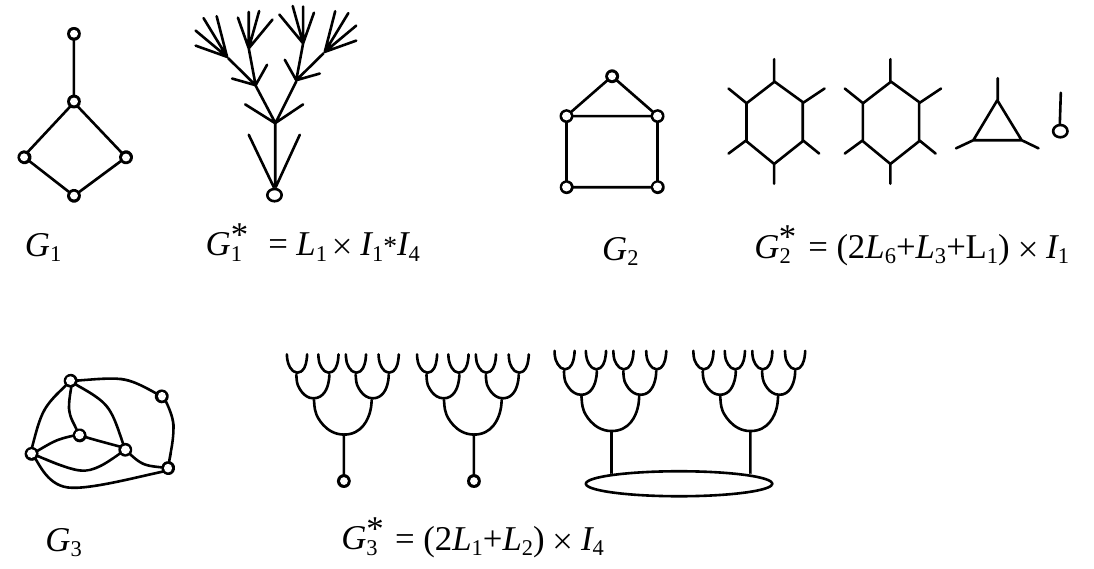}
\]
\caption{Three graphs and their evolution digraphs.}
\label{fig:threeG}
\end{figure}
%
%

It turns out that digraph $G^*$ consists of certain number of
cycles, (later collectively called the  {\it loop ensemble}), 
to each node of which attached is a copy of a descending tree  
(later called the {\it characteristic tree} of $G$).
Digraph $G^*$ is a source of invariants of $G$.
\\

Before we go on, inspect Figure \ref{fig:diG} and \ref{fig:threeG}
that present evolution digraphs for a few small graphs
(ignore for now the labels in Figure \ref{fig:threeG}).
Notice these features:
(i) the sum of nodes in the loops is a power of two;
(ii) the size of each loop divides the size of the longest loop;
(iii) each descending tree has the same structure;
(iv) the number of nodes in each tree is a power of two;
(v) trees have some regularity in shape. Thus the shape of the tree 
and of the loop part would suffice to reconstruct the digraph, as the arc 
orientations are unambiguous.

\section{Geometry of Graphs}

The following `geometrization of graphs' will be exploited.
The space of states $\VVV_G$ of a graph will be viewed as an
$n$-dimensional linear space over ${\bf Z}_2$, that is, $\VVV_G
\cong {\bf Z}_2^n $, where $n=|\hbox{Ver}(G)|$.
The structure of the graph is represented by two ${\bf Z}$-valued 
scalar products in the space $\VVV_G$. 
One is the natural scalar product $g$, which for the vectors representing 
single vertices is $g(v_i,v_j)=\delta_{ij}$. 
The other is a (possibly degenerated) scalar product $A$ represented 
in the natural basis by the adjacency matrix, that is 
$A(v_i,v_j) = 1$ if $(v_i,v_j)\in\hbox{Ver(G)}$, and $A(v_i,v_j) = 0$ otherwise.  
\\
\\
The harmonic operator (\ref{eq:a}) is geometrically an endomorphism of 
space $\VVV_G$, self-adjoint with respect to $g$.
One may easily show that in the natural basis of vertices, 
it has the matrix form  
$$
             a=A+D \ ,
$$
where $D$ is a diagonal matrix with values 1 at the diagonal 
entries that correspond to the odd vertices, and 0 otherwise. 
This is the Laplacian of the graph taken modulo 2.
As a matrix, $a$ has an even number of 1's in every row and every column.
Let $\Omega$ denote a column of all entries equal to one. Then actually 
any square matrix  with entries $0$ and $1$ only, 
satisfying (i) $a^T=a$ and (ii) $a\Omega = 0$, determines a graph. 

\bigskip

Evolution of a state $S\in\VVV$ is given by $S(t)=a^tS$ where $t\in{\bf N}$
plays the role of time. Thus, in order to study harmonic evolutions 
we need to start with the monoid of endomorphisms generated by $a$.  
Due to the finiteness of $G$, the monoid is of the form:
\[
       M(a)=\{1,\, a,\, a^2, \ldots , a^k, \dots, a^{k+n-1} \} \ ,
\]
where $a^{k+n}=a^k$ for some $k$ and $n$. 
Part $L(a)=:\{a^k,\ldots,a^{k+n-1}\}$ forms a cyclic
group, part $T(a)=:\{1,\ldots,a^k\}$ is called the {\it tail} of $M(a)$.
We shall use the notation $n=|L(a)|$ and $k=|T(a)|$.

\begin{lemma}
\label{le:1}
Operator $\pi  =a^p$ with $p=|L(a)|\cdot|T(a)|$ 
is an orthogonal projection in $\VVV$.
\end{lemma}

\begin{proof} 
One may easily show that $\pi$ is the neutral element of the group $L(a)$.
Therefore, $\pi^2=\pi$, i.e.,  $\pi$ is a projection.
Orthogonality of $\pi$ follows from self-adjointness of $a$.
\end{proof}

\begin{proposition}
\label{pr:2} 
The space of states $\VVV_G$ is a direct product of two subspaces
\begin{equation}
\label{eq:VTL}
         \VVV_G=T(G)\oplus L(G)
\end{equation}
where $T(G)={\rm Ker}\,\pi  $ and $L(G)= {\rm Im}\,\pi$.
Moreover, restriction of $a$ to $T(G)$ is nilpotent, $a^k=0$,
and restriction of $a$ to $L(G)$ is an automorphism of $\VVV$.
\end{proposition}

Lemma \ref{le:1} implies:

\begin{corollary}
The dimension ${\rm dim}\, L(G)$ is even for any graph $G$.
\end{corollary}

\begin{proof}
Indeed, projection $\pi  $ must have two eigenvalues:
namely 1 for the subspace $L(G)$, and 0 for $T(G)$.  Therefore,
a basis of $\VVV_G$ exists in which $\pi  $ is expressed by
a diagonal matrix \ $\pi  _0={\rm diag} [1,1,\ldots,0,0\ldots,0]$.
Trace does not depend on basis; therefore,
${\rm dim}\, L ={\rm Tr}\,\pi_0  ={\rm Tr}\,\pi$.
Since ${\rm Tr}\, a^i$ is even for any $i$, so is ${\rm dim}\, L$.
\end{proof}

Proposition \ref{pr:2} explains the general structure of evolution 
digraphs.  Here is its meaning.
\\

({\it i})
The $n$-cube $\VVV_G={\mathbf Z}_2^n$ may be considered as a digraph 
$(\VVV_G, \looparrowright)$,  where the ``looped arrow" denotes the 
relation of succession.  That is, for any two vectors (states) $x,y\in \VVV_G$
we say $x\looparrowright y$ only if $y=ax$. 
\\

({\it ii}) 
Subspace $T(G)\in \VVV_G$ is closed under (nilpotent) action of $a$.  
It may be, too, considered as a digraph 
$(T(G), \looparrowright)$,  with a similarly defined relation of succession.  
Due to nilpotency of $a$, this digraph has a form of a tree
with $0$ as the root. We shall call this digraph
the {\bf characteristic tree} of $G$.
\\

({\it iii})
Similarly, the subspace $L(G)\in \VVV_G$ is closed under 
action of $a$, but this time $a$  acts as an automorphism.  
Thus, as a digraph $(L(G), \looparrowright)$, 
it consists of a number of loops (cycles) arising as the orbits of
the action of the cyclic group generated by $a$.
We shall call this digraph the {\bf loop ensemble} of $G$.
\\
 
({\it iv})
Now, the evolution digraph $G^*$ results as a ``semidirect" 
product of the two,  
$$
G^* = L(G)\ltimes T(G)  = \{L(G)\times T(G), \looparrowright \}
$$
with the succession in $L\times T$ defined
$$
(x,y)\looparrowright (x',y') \hbox{\ if either  } 
(x=x' \hbox{\ and } y\looparrowright y')
\hbox{ \ or } 
(x\looparrowright x' \hbox{\ and } y=y'=0) \ .
$$
As a digraph, $G^* = T(G)\ltimes L(G)$ will be simply denoted 
as $G^* = \{T(G),L(G)\}$.  
\\

In the following two sections we shall characterize the two digraphs.

\section{Classification of Harmonic Trees.}

Let $A*B$ denote the digraph obtained as the direct product 
of two digraphs $A$ and $B$.  That is we set $(a,b)\prec(a',b')$ iff 
$a\prec b$ in $A$ and $b\prec b'$ in $B$.
\begin{theorem}
\label{th:3}
The characteristic tree of a graph can be uniquely 
factored into a product of binomial trees:
\begin{equation}
\label{eq:TII}
        T(G) = I^{b_1}_1 * I^{b_2}_2 * \ldots * I^{b_k}_k
\end{equation}
where $I_i$ denotes a binomial tree of height $i$.
\end{theorem}

\begin{proof}
Since  $a$ restricted to $T(G)$ is nilpotent, then ---by Jordan decomposition
theorem--- there exists a basis in $T(G)$ such that operator $a$ takes in it 
a quasi-diagonal form:
\begin{equation}
\label{eq:aJ}    
    [a]=\left[\begin{array}{ccccc}
                       J_1&   &      &    &        \cr
                          &J_2&      &    &        \cr
                          &   &\ddots&    &        \cr
                          &   &      &J_i &        \cr
                          &   &      &    &\ddots  \end{array}\right]
\qquad \hbox{where} \quad
    J_i=\left[\begin{array}{ccccc}
                         0&   &   &       &      \cr
                         1& 0 &   &       &      \cr
                          & 1 & 0 &       &      \cr
                          &   &   &\ddots &      \cr
                          &   &   &  1    & 0    \end{array}\right]
\end{equation}

(possibly, some $J_i$ are $1\times 1$ null matrices).
This means that the space $T(G)$ decomposes into eigenspaces of
the harmonic operator, on which map $a$  acts as a ``raising'' operator.
Namely, denote the vectors that span the subspace $E_i$ corresponding to the
sub-matrix $J_i$ by $\{f_1,\ldots, f_q\}$. Than $a$ acts :
\begin{equation}
\label{eq:ff}
           f_i \; \xrightarrow{\ a\ }\; f_{i+1}
           \qquad\hbox{\rm and}\qquad
           f_q \; \xrightarrow{\ a\ } \; 0
\end{equation}
Note that each $f_i$ (for $i>1$) is an image of exactly two states,
namely $f_{i-1}$ and $f_q + f_{i-1}$. So is, therefore, any sum of these
vectors. The state $f_1$, and any sum $\sum_{i\in I}\,f_i$ containing $f_1$,
do not have any anti-images through $a$.
Therefore, the states of the subspace $E_i$ form a binomial tree of
height $q$ with $2^q$ edges, among which half, $2^{q-1}$, form the top
(starting) nodes. Decomposition  of $a$ into linearly independent blocks (\ref{eq:aJ})
corresponds to multiplication (\ref{eq:TII}) of such trees.
\end{proof}

\begin{figure}
\[
  \includegraphics[width=4.5in]{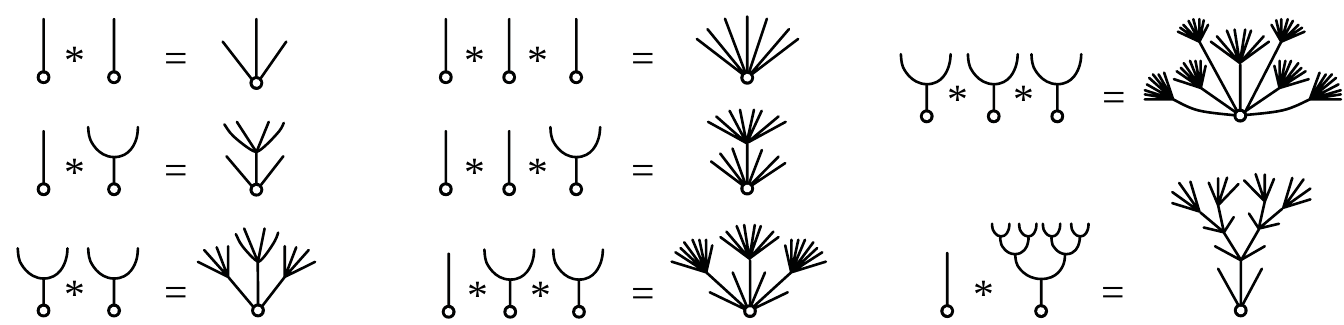}
\]
\caption{Products of binary trees: a fragment of multiplication table.}
\label{fig:TTT}
\end{figure}

A few products of binary trees are shown in Fig. \ref{fig:TTT}. Note that 
the height of the resulting tree is that of the highest tree in the product.
\\

The above theorem implies that
the number of the binomial trees in the decomposition (\ref{eq:TII}) 
equals ${\rm dim}\, {\rm Ker}\,a$.
Indeed, in the quasi-diagonal form (\ref{eq:aJ}) of the harmonic matrix, each
$J_i$ contributes a single one-dimensional subspace to ${\rm Ker}\,a$.
(Clearly ${\rm dim}\,{\rm Ker}\,a\,|_{T(G)} = {\rm dim}\,{\rm Ker}\,a$,
since ${\rm Ker}\, a\subset T(G)$).
Another implication is that if ${\rm dim}\, {\rm Ker}\, a=1$, then $T(G)$
is a binomial tree, $T(G)=I_{|T(a)|}$. If ${\rm dim}\,{\rm Ker}\, a=2$,
then  $T(G)=I_{|T(a)|}*I_{{\rm dim}\, T(G)-|T(a)|}$.
\\

The composition of the characteristic tree, that is the exponents of the 
factorization (\ref{eq:TII}), may be deduced by means of the elements 
of the monoid $M(a)$.

\begin{theorem}
\label{th:4} 
The multiplicity of the binomial tree $I_j$ in $T(G)$ is given by the 
following formula
\begin{equation}   
\label{eq:nKK}
    n_j =  2{\rm dim}\,{\rm Ker}\, a^j
           -{\rm dim}\,{\rm Ker}\, a^{j+1}
           -{\rm dim}\,{\rm Ker}\, a^{j-1}
                                    \qquad\hbox{for $j\leq k$}
\end{equation}
\end{theorem}

\begin{proof}
The dimensions of kernel of powers of $a$ satisfy the following system of equations:
\begin{equation}
\label{eq:KKK}
\begin{array}{ccl}
  {\rm dim}\, {\rm Ker}\, a^{\phantom{2}}
                              &=& n_1+n_2+\ldots+n_k             \cr
  {\rm dim}\, {\rm Ker}\, a^2 &=& n_1+2(n_2+\ldots+n_k)          \cr
  {\rm dim}\, {\rm Ker}\, a^3 &=& n_1+2n_2+3(n_3+n_4+\ldots+n_k) \cr
   \ldots       & &                                              \cr
  {\rm dim}\, {\rm Ker}\, a^k &=& n_1+2n_2+3n_3+\ldots+k n_k     
\end{array}
\end{equation}
Denote $K_i={\rm dim}\,{\rm Ker}\, a^i$. Then (\ref{eq:KKK}) is equivalent to
a matrix equation ${\bf K} = M {\bf n}$, where $M_{ij}=\min (i,j)$.
The inverse matrix $M^{-1}$, that is easy to find (here for k=5 for illustration):
\begin{equation*}
    K =\left[\begin{smallmatrix}
                         1& 1& 1& 1& 1 \cr
                         1& 2& 2& 2& 2 \cr
                         1& 2& 3& 3& 3 \cr
                         1& 2& 3& 4& 4 \cr
                         1& 2& 3& 4& 5 \end{smallmatrix}\right]
\qquad \Longrightarrow\quad
    K^{-1}=\left[\begin{smallmatrix}
                \hfill  2 &\hfill -1 &\hfill  0 &\hfill  0 & \ 0 \cr
                \hfill -1 &\hfill  2 &\hfill -1 &\hfill  0 & \ 0 \cr
                \hfill  0 &\hfill -1 &\hfill  2 &\hfill -1 & \ 0 \cr
                \hfill  0 &\hfill  0 &\hfill -1 &\hfill  2 & \ 0 \cr
                \hfill  0 &\hfill  0 &\hfill  0 &\hfill -1 & \ 2 \end{smallmatrix}\right] \ ,
\end{equation*}
solves (\ref{eq:KKK}) into (\ref{eq:nKK}), as claimed.
\end{proof}

This theorem gives the algorithm for retrieving the structure of the characteristic
tree from the harmonic matrix $a$ of a given graph.

\section{Harmonic Loops} 

Now we shall look for the structure of the loop ensemble of a given graph.
The elements of group $L(a)$ act --- if restricted to the subspace $L(G)\in\VVV_G$
 --- as automorphisms. The action of $L(a)$ decomposes the subspace $L(G)$ into a
number of orbits. Since the group $L(a)$ is cyclic, the orbits are loops.
(Clearly, the orbits do not form linear subspaces of $L(G)$).
It follows thus:

\begin{proposition}
\label{pr:5} 
The least common multiple of the lengths of
the loops in $L(G)$ is equal to the order of the loop group,
$m=|L(a)|$.
There is a loop in $L(G)$ of length $m$.
In particular, the rank of any vector of $L(G)$ divides $m$.
\end{proposition}

Denote a loop of length $i$ by $L_i$.
If $nL_i$ denotes a disjoint sum of $n$ loops of order $i$,
then the decomposition of $L(G)$ into loops may be expressed
symbolically as a formal sum
\begin{equation}
\label{eq:LnL}
               L(G) = n_1L_1 + n_2L_2 + \ldots + n_mL_m
\end{equation}
(for examples this notation see Fig. \ref{fig:threeG})
Due to Proposition \ref{pr:5}, the only nonzero numbers $n_i$ could be 
those with $i$ dividing $m=|L(a)|$.  The total number of loops is $n=\sum n_i$.
Clearly all lengths of the loops sum up to a power of two,
namely $\sum_i {i\cdot n_i} = |L(G)|=2^{{\rm dim}\, L(G)}$. 
\\

Now we recover the exact structure of the loop ensemble.  That is,
we retrieve the coefficients of the decomposition (\ref{eq:LnL}), 
from the monoid $M(a)$.  
We shall denote  the eigenspace of an endomorphism $g$ by $F(g)$, and 
its dimension by ${\rm dim}\, F(g)$.
\begin{theorem}
\label{th:6}
The number of loops of length $p$ in $L(G)$ is
\begin{equation}
\label{eq:nmu}
     n_p = \frac{1}{p}\sum\limits_{i|p} 
                      \mu\left( {\frac p i} \right)
                   \;2^{{\rm dim}\, F(a^i)}
\end{equation}
where $\mu$ denotes Moebius function, and the sum runs over divisors of $p$.
\end{theorem}

\begin{proof}
First, notice that $F(a^i) \subset F(a^j)$ \ only if \ $i|j$.
Hence, for any $j$ we have
\begin{equation}
\label{eq:Fmu}
    |F(a^j)| = \sum_{d|j}^j d\cdot n_d
             = \sum_{d=1}^j d\cdot n_d\cdot \mu(d,j)
\end{equation}
where $\mu$ is the ``factorization function'' defined: 
$\mu(i,j)=1$ if $i|j$, and $\mu(i,j)=0$ otherwise.
For conciseness, denote   $F_i=|F(a^i)|= 2^{{\rm dim}\, F(a^i)}$
(this is the number of states in the eigenspace of $a^i$).
The system of equations (\ref{eq:Fmu}) has form ${\bf F} = A {\bf n}$, 
where $A$ is a lower-triangular matrix
\begin{equation}
\label{eq:AA}
  A = \left[\begin{array}{ccccccc}
                    1 &   &   &   &   &   &     \cr
                    1 & 2 &   &   &   &   &     \cr
                    1 & 0 & 3 &   &   &   &     \cr
                    1 & 2 & 0 & 4 &   &   &     \cr
                    1 & 0 & 0 & 0 & 5 &   &     \cr
                    1 & 2 & 3 & 0 & 0 & 6 &     \cr
                      &   &   &   &   &   &\ddots  \end{array}\right]
     \qquad\hbox{or}\qquad
     A_{ij} = i\cdot\mu(i,j)
\end{equation}
It is well-known that the inverse of matrix $M_{ij}=\mu(i,j)$ is given by
$(M^{-1})_{ij}= \mu({\frac i j})$, where $\mu$ is the Moebius function,
defined
$$
\mu(d) = \left\{\begin{array}{cl}
              1 &\hbox{\ if  } d=1 \cr
         (-1)^r &\hbox{\ if  } d=p_1\cdot p_2\cdot\ldots\cdot p_r\cr
           0    &\hbox{\ if  } p^2|d \hbox{\ for any prime } p 
                \end{array}\right.
$$
where $p_1, p_2,\ldots,p_r$ are distinct primes 
(for the properties of the Moebius function see \cite{Rot}).
Equation (\ref{eq:AA}) is a modified version of the factorization function, 
and it easily solves to (\ref{eq:nmu}).
\end{proof}

The above theorem provides an algorithm to construct the structure of 
the loop ensemble of the evolution digraph of a given graph.

\begin{figure}
\[
  \includegraphics[width=4.5in]{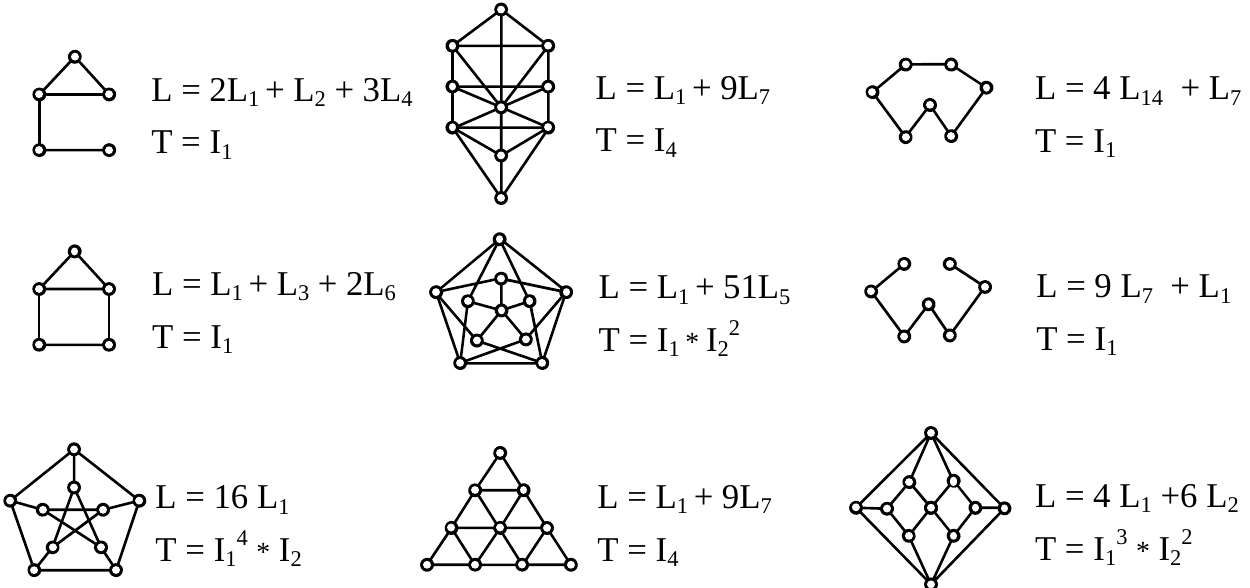} %
\]
\caption{Examples of graphs and their loop and tree structure.}
\label{fig:EEE}
\end{figure}

%
\begin{remark} \rm
Here is alternative view on the inversion (\ref{eq:nmu}) of the system (\ref{eq:Fmu}).
Let $\{\N,\prec\}$ be the poset (lattice) of natural numbers with
the order induced by multiplicative relations,
i.e. $a\prec b$ if $a|b$.
Let $[a]$ denote the sub-poset of all elements below $a$, i.e.
$[a]=\{n\in\N \;|\; n\prec a\}$  (a filter generated by $a$).
Let $\overline{a}$ be the set of maximal elements in $[a]-\{a\}$; 
say $\overline {a}=\{a_1,\ldots ,a_n\}$.
The subsets $[a_1]$, $[a_2]$,...,$[a_n]$ form  a ``$\cap$-algebra'', i.e.
their  mutual intersections are again some subposets
generated by elements of $[a]$.
\def\xxx{\!\!\!\diagdown\!\!\!\!\!\diagup\!\!\!}
\def\tiagup{\!\!\!\diagup\!\!\!}
\def\tiagdown{\!\!\!\diagdown\!\!\!}
$$
\begin{array}{ccccc}
        &          & 30 \\ 
        &\tiagup   &    &\tiagdown \\ 
     10 &          &    &          &15  \\     
     |  &\tiagdown &    &\tiagup   & |  \\
     2  &          & 5  &          &3   \\
        &\tiagdown & |  &\tiagup   &    \\
        &          & 1  &          &    
\end{array}
\qquad\quad
\begin{array}{ccccccc}
           &         &    &        & 60                  \\ 
           &         &    &\tiagup &  | &\tiagdown       \\
           &         & 30 &        &20  &           &12  \\ 
           &\tiagup  &  | &\xxx    & |  &\xxx       & |  \\
        15 &         & 10 &        & 6  &           & 4  \\ 
        |  &\xxx     &  | &\xxx    & |  &\tiagup         \\
        5  &         &  3 &        & 2                   \\ 
           &\tiagdown&  | &\tiagup                       \\
           &         &  1        
\end{array}
$$
The numbers of vertices in loops of order $k$ may be  obtained
by the inclusion-exclusion principle applied to
the system of subsets generated by $\overline k$.
For example: $n_{30}= F_{30}-F_{10}-F_{15}+F_{5}$\ and \
$n_{60}= F_{60}-F_{30}-F_{20}-F_{12}+F_{10}+F_{6}+F_{4}-F_{2}$.
\end{remark}


\section{Disjoint sum of graphs} 

The results of the previous sections may be formulated as follows.
On the one hand we have a category $\cal{G}$ of finite graphs. On the other 
we have the category of binary-generated trees $\cal{T}$ and the category
of loop ensembles $\cal{L}$:
\begin{equation}
\begin{array}{rl}
{\cal T} =& \{\;\prod\limits_i\;I^{n_i}_i\;|\; i,n_i\in{\bf N}\;\} \\
{\cal L}= &\{\;\sum\limits_i n_iL_i \;|\; i, n_i \in{\bf N}\;\}
\end{array}
\end{equation}
The construction described in this paper 
is a map
\begin{equation}
\label{eq:GLT}
                   *:\ {\cal G} \  \longrightarrow \ {\cal L}\times {\cal T}         
\end{equation}
which associates to each graph $G$ two elements, $L(G)\in{\cal L}$
and $T(G)\in{\cal T}$.
The digraph of evolution $G^*$ is a semidirect product of these two,
namely, if $\{L, \looparrowright \}$ and $\{T, \looparrowright \}$ are viewed as digraphs
$$
G^* = \{L\times T, \looparrowright \} = L\ltimes T
$$

The problem is to find how various operations and maps in
the category of graphs are reflected in the algebraic properties
of the spaces ${\cal L}$ and ${\cal T}$.
\\

From many such problems we answer a modest question concerning 
the evolution graph for a disjoint sum of two graphs, $G_1\cup G_2$.
Before we go on, a few remarks on the two categories: 
of trees and of loop ensembles.  

~\\
{\bf Universal family of binary generated trees.}
The family of trees 
${\cal T} := \hbox{gen\,} \{I_0,I_1,I_2,\ldots\}$
consists of all finite $*$-products of binomial trees.
It forms a semigroup $({\cal T}, *)$ satisfying:

\begin{equation}
\label{eq:TTTT}
\begin{array}{ll}
1) &T_1 * T_2 = T_2 * T_1\\
2) &(T_1 * T_2) * T_3 = T_1 * (T_2 * T_3) \\
3) &T * I_0 =  T   
\end{array}
\end{equation}
%
Thus $({\cal T}, * )$ forms a commutative free monoid
with exponentials from $\N$.
\\
\\
{\bf Universal loop space.} 
Similarly, let us introduce a ``space'' ${\cal L}$ over $\N$ spanned by the 
the set $\{L_1, L_2,\ldots\}$, where $L_i$ represents a loop length $i$.
Using additive notation, a collection of loops (loop ensemble) will be 
denoted as a formal sum, e.g. $L=L_3 + 2L_4=\{ \bigtriangleup,\square,\square\}$.
Now we shall turn ${\cal L}$ into an algebra by introducing a product 
$*:\;{\cal L}\times{\cal L}\to{\cal L}$,  
which for single loops $L_a$ and $L_b$ of length $a$ and $b$
respectively is defined
\begin{equation}
\label{eq:LL}
L_a * L_b = \langle a|b \rangle \ L_{[a|b]}   
\end{equation}
where we denote $\langle a|b\rangle =\hbox{gcd\,}(a,b)$ and $[a|b] =\hbox{lcm\,}(a,b)$.
\begin{theorem}
The triple $\{{\cal L}, +, * \}$ satisfies
$$
\begin{array}{llll}
1) &L_a * L_b = L_b * L_a                  &1')& L_a + L_b = L_b + L_a\\
2) &(L_a * L_b) * L_c = L_a * (L_b * L_c)  &2')& (L_a + L_b) + L_c = L_a + (L_b + L_c)\\
3) &L_1 * L =  L \ \ \hbox{\rm for any  } L   &3')& L_0 + L =  L \ \ \hbox{\rm for any  } L \\
4) &(L_a + L_b) * L_c = L_a * L_c + L_b * L_c   
\end{array}
$$
\end{theorem}

\begin{proof}Let us show item 2):
Recall that $\langle a|b\rangle = {\frac{a\cdot b}{[a|b]}}$.
Using (\ref{eq:LL}) we get
$(L_a * L_b) * L_c = {\frac{a\cdot b}{[a|b]}  }\ L_{[a|b]} * L_c =
{\frac{a\cdot b}{[a|b]}}\cdot {\frac{[a|b]\cdot c}{[a|b|c]}}\ L_{[a|b|c]}=
{\frac {a\cdot b \cdot c}{[a|b|c]}}\ L_{[a|b|c]}=
\langle a|b|c \rangle\ L_{[a|b|c]}$.
Since $[\ |\ ]$ is an associative operation, property (2) holds.
\end{proof}

The definition in Eq. \ref{eq:LL} is chosen to
make the following look simple.

\begin{proposition} 
\label{co:1} 
The evolution digraph of a disjoint union of two graphs is
\begin{equation}
\label{eq:GGG}
            (G_1+G_2)^* = \{T(G_1)*T(G_2), \; L(G_1)*L(G_2)\}     
\end{equation}
where the star operations are defined in (\ref{eq:LL}) for loops, 
and (\ref{eq:LL}) for the characteristic trees.
\end{proposition}

\section{Power Graphs} 

\begin{definition}
The $p^{\hbox{th}}$ {\it power} of a graph $G$ is the graph $G^p$ 
determined by the $p^{\hbox{th}}$ power of the harmonic matrix of graph $G$, $a^p$.
\end{definition}

Clearly, $G^p$ has the same set of vertices as $G$, but a different 
structure of edges.
That the powers of a graph are well-defined is obvious, since  
$a^p\Omega=a^{p-1}(a\Omega)=0$ and $(a^p)^T=(a^T)^p= a^p$.
An interesting question is to determine the evolution digraph 
of graph $G^p$ from the the evolution digraph of $G$ alone.
\\

From the basic properties of the monoid $M(a)$ we obtain
the length of the tail and rank of the cycle group of $M(a^p)$ to be
\begin{equation}
\label{eq:tail-loop}
\begin{array}{rl}
\hbox{tail:~~} &\quad |T(a^p)| = \hbox{Ent\,}\bigl( |T(a)|/q\bigr)  \\
\hbox{loop:~~} &\quad |L(a^p)| = |L(a)|\;\fd\; q                    \; ,
\end{array}
\end{equation}
where $\fd$ denotes ``funny division,'' defined for integers by
$a\fd b = [a|b]/b$ ~ (recall that $[a|b]=\hbox{lcm\,}(a,b)$).
For instance\
$6\fd 1=6$, \ $6\fd 2=3$, \ $6\fd 3=2$, \ $6\fd 4=3$, \
$6\fd 5=6$, \ $6\fd 6=1$.

\begin{proposition}
\label{pr:power}
Let the evolution digraph $G^*$ of graph $G$ have the following 
characteristic tree and the loop ensemble:
\begin{equation*}
\label{eq:propTL1}
\begin{array}{rl}
T(G)&= I_1^{b_1}*I_2^{b_2}*\ldots *I_k^{b_k}=
                        \prod\limits_{i=1}^k \ I_i^{b_i}\cr
L(G)&= c_1L_1+c_2L_2+ \ldots +c_m L_m=
                        \sum\limits_{i=1}^m \ c_iL_i  \ ,
\end{array} 
\end{equation*}
where $k=|T(a)|$ and $m=|L(a)|$.
Then the evolution digraph $(G^q)^*$ consists of
\begin{equation}
\label{eq:propTL2}
  \begin{array}{cl}
     T(G^q) &= \prod\limits_{i=1}^{k'}\ I_i^{w_i} 
               \qquad\hbox{where}\quad
                w_i=\sum\limits_{j=q(i-1)+1}^{q\cdot i} \ j\cdot b_j \\
     L(G^q) &= \sum\limits_{i=1}^{m'} \ c_i\cdot\langle i|q\rangle\ L_{i\fd q}
\end{array}
\end{equation}
where $k'=|T(a^p|$ and $m'=|T(a^p)|$ are those of (\ref{eq:tail-loop}) above.
\end{proposition}

\begin{proof}
The proof of this is simpler in combinatorial terms.
Express $G^*$ as the product of the loop ensemble and the characteristic tree.
A state $S$ is equivalent to a choice of a node in one of the loops in
$L(G)$, and a node in each binary tree constituting $T(G)$ 
(imagine them as ``glowing points'').
The action of the harmonic operator on $S$ is represented now
by a shift of each glowing point by one position:
along the loop, and down each tree.
Tracing their $q$-step movements brings the conclusion to the proof.
\end{proof}

{\bf Remark:} Every integer $q\in\mathbf{N}$ defines a linear map on 
the universal loop space 
$
     {\widehat q}:\ {\cal L} \ \longrightarrow \ {\cal L}              
$
which transforms each single loop $L_a$ into a sum of loops
according to
$$
               L_a \ \longrightarrow \ \langle a|q\rangle\; L_{a\fd q}      
$$
as suggested by (\ref{eq:propTL2}).
For instance $\widehat 4 (L_6) = 2L_3$, which is due to the {\it two} distinct 
paths of ``4-step jumps" in a hexagon.
The second part of Proposition \ref{pr:power} may be now expressed simply
$$ 
                       L(G^q) = \widehat q {L(G)}               
$$
Notice that operator $\widehat{\ }$ is a representation of natural 
numbers with multipliation, namely 
$\widehat a \widehat b = \widehat{a\cdot b}$.
On the other hand, we have the {\it adjoint} representation of the 
algebra $\{{\cal L},*\}$, where for $a\in\N$ we define   \
$
{\mathop a\limits^*} \; L_i = L_a * L_i
$.
Each of these representations is of course commutative, but the common algebra
of both kinds of operators is not.
For instance:
%
\def\dwas{\mathop 2\limits^*}   
\def\dwah{\mathop 2\limits^{\wedge}} 
$$
\dwas\dwah L_i =\left\{\begin{array}{ll}
                     L_{2i}  & \hbox{if } i= \hbox{ odd} \cr
                     2L_i    & \hbox{if } 2|i \hbox{ but } 4\!\not|i \cr
                     4L_{i/2}& \hbox{if } 4|i 
                  \end{array}\right.
\quad \hbox{whereas}\quad
\dwah\dwas L_i =\left\{ \begin{array}{ll}
                     2L_i    & \hbox{if }i= \hbox{ odd }\cr
                     4L_{i/2}& \hbox{if }i= \hbox{ even} 
                   \end{array}\right. 
$$

\section{Conclusion} 

We defined the harmonic operator of a graph as a topologically 
motivated endomorphism of the space of states of graphs.  The states 
(co-states) are meant as vectors in the linear space over $\mathbf{Z}$ 
formally spanned by the vertices (edges).  The paper concentrates on the former.
We defined a dynamical system on a graph as a recursive application 
of the harmonic operator.    
We found the types of evolutions (evolution digraphs) 
and classified them.  
An algorithm to foretell the evolution digraph for a particular graph 
has consists of Theorems \ref{th:4} and \ref{th:6}.
\\

We see it as a map
\begin{equation}
\label{eq:GLT2}
{\cal G} \  \longrightarrow \ {\cal L}\times {\cal T}         
\end{equation}
which associates to each graph $G$ two vectors, $L(G)\in{\cal L}$
and $T(G)\in{\cal T}$.
Trees and loops correspond to structure of invariant subspaces of $a$;
any number derived from it may be viewed as an invariant of the
graph.
\\

Map (\ref{eq:GLT2}) is a morphism of objects, and the correspondence for
operations of the disjoint union and for the graph power were shown.
The problem of finding such a correspondence for other operations on graphs
may prove to be rather difficult.
\\

The inverse problem asks whether a given digraph represents a 
harmonic evolution of some graph.
Equivalently, the problem is to characterize the images
 $L({\cal G})\subset{\cal L}$ and $T({\cal G})\subset{\cal T}$.
We offer these two conjectures.

\bigskip
\noindent
{\bf Conjecture 1:} \ {\sl
Any finite product of binary trees of height at least 1 is a
characteristic tree for some graph.
}

\bigskip
{\bf Conjecture 2:} \ {\sl
Denote ${\cal H} = \{2^i\;|\;i\in{\bf N}\}=\{1,2,4,8,\ldots\}$ 
(the set of ``highly even numbers"). 
A formal sum $L=\sum n_iL_i$ is said to be 
{\rm admissible}, 
$L\in {\cal L}_{A}$, if the following is satisfied:
\begin{eqnarray*}
       1)&& n_1\in{\cal H}  \cr
       2)&& \sum_{i|p} i\cdot n_i \ \in\ {\cal H} \hbox{~~for each~~} p
                                         \hbox{~~such that~~} n_p\not= 0 \cr
       3)&& \hbox{if~~} n_i\not=0 \hbox{~~or~~} n_j\not=0
                                       \hbox{, ~~then~~} n_{[i|j]}\not=0 \cr
       4)&& \log_2\sum_{i} i\cdot n_i \hbox{~~~is even} \cr
\end{eqnarray*}
where we denote $[i|j]=\hbox{\rm l.c.m.}\, (i,j)$.
Then the set of harmonic loop ensembles of finite graphs 
coincides with the above admissible set, ${\cal L}({\cal G})={\cal L}_{A}$.
}
\\

Further exercises and problems:  1. rewrite the theory for 
the co-harmonic operator $b=\delta\partial$ acting on the edges and 
defining the harmonic "co-evolution" on graphs.
Can one determine the corresponding co-evolution digraph from the 
evolution graph alone? 
2. Replace the field $\mathbf{Z}_2$  
by another commutative ring (like $\mathbf{Z}_3$).
3. Classify the rules for evolution on graphs other than harmonic rules.

\vskip0.2in

\noindent
{\bf {\large Acknowledgements}}
\\
\\
The author would like to thank Philip Feinsilver and Greg Budzban
for fruitfull conversations and for their interest in this project.



\begin{thebibliography}{10}


\bibitem{AM}%
W. N. Anderson, Jr. and TD Morley,
Eigenvalues of the Laplacian of a graph,
{\it Linear and Multilinear Algebra} {\bf 18} (1985) 141--145.

\bibitem{GR}%
Chris Godsil and Gordon Royle,
{\sl Algebraic Graph Theory}, Springer, 2000.

\bibitem{Pri}%
Erich Prisner,
Graph Dynamics.  
(Longman, Essex, 1995).

\bibitem{Rot}%
Gian-Carlo Rota,
{\sl On the Foundations of Combinatorial Theory,
I. Theory of Moebius Function},
Z. Wahrscheinlichkeitstheorie {\bf 2} (1964) 340--368.

\bibitem{Tut}%
W. T. Tutte,
{\sl Graph Theory} (Addison-Wesley, Reading, 1984)~
(as Volume 21 of {\sl Encyclopedia of Mathematics and its Applications},
G.-C. Rota, ed.).

\bibitem{Whe1}%
John Archibald Wheler,
{\sl Frontiers of Time}  
(North-Holland, Amsterdam, 1979).

\bibitem{Whe2}%
John Archibald Wheler,
Pregeometry: Motivations and prospects.  
In {\sl Quantum theory and Gravitation} (Springerm, 1980).

\end{thebibliography}
\end{document}